\documentclass[reqno,11pt]{amsart}

\allowdisplaybreaks

\usepackage[truedimen,margin=25truemm]{geometry}
\usepackage{amsmath,amsfonts,amssymb,graphicx,amsthm,color,yfonts,cite, latexsym}
\usepackage{paralist}
\usepackage{mathrsfs}
\usepackage{hyperref}
\usepackage{physics}
\usepackage{bbm}
\usepackage{extarrows}
\usepackage{bm}
\usepackage{mathtools}
\usepackage{orcidlink}
\mathtoolsset{showonlyrefs}

\usepackage{tikz}
\usepackage{tikz-cd}
\usetikzlibrary{cd}
\usetikzlibrary{arrows.meta}
\usetikzlibrary{matrix, arrows, positioning}
\usetikzlibrary{shadows}

\newtheorem{theorem}{Theorem}
\newtheorem{definition}[theorem]{Definition}
\newtheorem{proposition}[theorem]{Proposition}

\newtheorem{lemma}[theorem]{Lemma}
\newtheorem{corollary}[theorem]{Corollary}

\newtheorem{remark}[theorem]{Remark}

\makeatletter
\newcommand*{\rom}[1]{\expandafter\@slowromancap\romannumeral #1@}
\makeatother

\newcommand{\ls}{\lesssim}

\newcommand{\R}{\mathbb{R}}

\newcommand{\C}{\mathbb{C}}

\newcommand{\CB}{\mathcal{B}}

\newcommand{\CF}{\mathcal{F}}

\newcommand{\SA}{\mathscr{A}}
\newcommand{\SB}{\mathscr{B}}

\newcommand{\de}{\delta}
\newcommand{\ps}{\psi}
\newcommand{\ph}{\varphi}
\newcommand{\ta}{\tau}
\renewcommand{\th}{\theta}

\newcommand{\et}{\eta}

\newcommand{\Ph}{\Phi}

\newcommand{\pl}{\partial}
\newcommand{\wt}{\widetilde}
\newcommand{\wh}{\widehat}
\newcommand{\loc}{{\rm loc}}

\newcommand{\Ck}[1]{\left\{#1\right\}}

\newcommand{\Dk}[1]{\left[#1\right]}
\newcommand{\K}[1]{\left(#1\right)}

\newcommand{\No}[1]{\left\| #1 \right\|}

\newcommand{\I}{\infty}
\newcommand{\sd}{\langle \pl_x \rangle}

\newcommand{\lxr}{\langle \xi \rangle}

\newcommand{\tw}{\frac{1}{2}}

\newcommand{\ov}{\overline}

\newcommand{\II}{\mathbbm{1}}

\newcommand{\Id}{\textup{Id}}

\renewcommand{\Im}{\operatorname{Im}}

\newcommand{\dH}{\dot{H}}
\newcommand{\HS}{\textup{HS}}

\makeatletter
\def\l@section{\@tocline{1}{0pt}{0pt}{}{}}%

\def\l@subsection{\@tocline{2}{0pt}{2.5em}{}{}}%

\def\l@subsubsection{\@tocline{3}{0pt}{3.0em}{}{}}%

\usepackage{wrapfig}
\usepackage{tikz}
\usetikzlibrary{arrows,calc,decorations.pathreplacing}
\definecolor{light-gray1}{gray}{0.90}
\definecolor{light-gray2}{gray}{0.80}
\definecolor{light-gray3}{gray}{0.60}

\numberwithin{equation}{section}

\numberwithin{theorem}{section}

\numberwithin{table}{section}

\numberwithin{figure}{section}
\title[Well-posedness for nonlocal NLS]{Well-posedness in $L^2$ for nonlocal NLS in the full subcritical range of symbol growth}

\date{\today}

\author[S. Hadama]{Sonae Hadama\, \orcidlink{0009-0007-1091-2520}}
\address{The University of Osaka}
\email{hadama.sonae.sci@osaka-u.ac.jp}
\thanks{The author would like to thank Professor Nobu Kishimoto for many helpful comments and suggestions through valuable discussions.
The author was supported by JSPS KAKENHI Grant Number 24KJ1338.}
\date{}

\begin{document}
\begin{abstract}		
In this paper, we study a class of one-dimensional nonlocal nonlinear Schr\"odinger equations on the line with nonlinearity given by a Fourier multiplier whose symbol has subcritical high-frequency growth. In terms of symbol order, this class is intermediate between the cubic nonlinear Schr\"odinger equation and the Calogero--Moser derivative nonlinear Schr\"odinger equation. We prove local well-posedness in $L^2(\mathbb{R})$ throughout the full scaling-subcritical range. Due to derivative loss, the standard Duhamel integral is not directly meaningful for rough data. To avoid this problem, we first construct the propagator $S_V$ for rough time-dependent potentials $V$, and then prove an Ozawa-Tsutsumi type bilinear Strichartz estimate for the perturbed flow $S_V$. 
These linear theories yield a concrete construction of rough solutions without using any equation-specific algebraic structure. For real-valued symbols, mass is conserved, and the local solutions are therefore global.
\end{abstract}
	
\maketitle
\tableofcontents
	
\section{Introduction}
In this paper, we consider well-posedness of the ``derivative'' nonlinear Schr\"odinger equation
\begin{equation}\label{eq:NLS}\tag{NLS}
	i\pl_t u + \pl_x^2 u = (a(D)|u|^2)u, \quad u:[0,\I)\times\R_x\to\C,
\end{equation}
where $a(D)$ is the Fourier multiplier with symbol $a:\R \to \C$.
Throughout this paper, we assume 
\begin{equation}\label{eq:assmption for m}
	\boxed{\sup_{\xi \in\R} \frac{|a(\xi)|}{\lxr^{1-\de}} <\I}
\end{equation}
for some $0\le \de \le 1$.
This class includes two important examples as endpoint cases.

On the one hand, when $\de=1$ and $a(\xi)=\pm 1$, then \eqref{eq:NLS} is the standard cubic nonlinear Schr\"odinger equation (cubic NLS).
There is a vast body of literature on the cubic NLS, but the readers can find some basic well-posedness results in \cite{Cazenave textbook}. For long-time behaviors of solutions, see the survey \cite{Murphy 2021} and references therein.

On the other hand, the scaling-critical case when $\de=0$ and $a(\xi)=\pm (|\xi|+\xi)$ corresponds to the Calogero--Moser derivative NLS (CM-DNLS).
For the Cauchy problem on the line with decaying data, previous well-posedness results for CM-DNLS and the related intermediate NLS were obtained in Hardy-type spaces in \cite{Gerard Lenzmann 2024, Killip et al 2025}, and in Sobolev spaces for the related intermediate NLS in \cite{de Moura Pilod 2010, Chapouto et al 2025+,Chapouto et al 2026+,de Moura 2007}. We also mention related results in other settings, including the periodic problem \cite{Badreddine 2024} and nonvanishing boundary conditions \cite{Akahori et al 2026+, X.Chen 2026}.

Note that previous works on the scaling-critical case $\de=0$ make essential use of equation-specific structures.
In contrast, the standard proof of the well-posedness for cubic NLS in $L^2(\R)$ is flexible; it can be applied for a wide class of nonlinearities, which is typical in the scaling-subcritical regime.

In this paper, we consider the intermediate cases of these two endpoints.
More precisely, we show that one can treat all scaling-subcritical nonlinearities $a(D)(|u|^2)u$ with symbol $a:\R\to \C$ satisfying \eqref{eq:assmption for m} and 
\begin{equation}
	\boxed{0<\de<1}
\end{equation}
in a unified way.
Since we can take $\de$ arbitrarily close to $0$, we can regard our result as covering the full scaling-subcritical range.
For the critical case $\de=0$, see the author's next work \cite{Hadama 2026+}.
This result seems natural because of the subcritical assumption; but it is also slightly surprising because the nonlinearity includes some derivative loss. Our analysis is based only on derivative gain and loss considerations, rather than on the special (algebraic) structure of individual models.

Now, we state our main result.

\begin{theorem}\label{th:main}
	Assume that $0< \de <1$ and $a:\R\to \C$ satisfies \eqref{eq:assmption for m}.
	Then, for any $\phi\in L^2(\R)$, there exist $T=T(a,\de,\|\phi\|_{L^2_x(\R)})>0$ and a unique solution $u(t)\in C([0,T];L^2_x(\R))$ to \eqref{eq:NLS} with initial condition $u(0)=\phi$ such that $|u(t)|^2 \in L^2_t([0,T];H^{1/2}_x(\R))$ in the sense of Definition \ref{def:solution} (see also Remark \ref{rmk:def of solution}).
	Moreover, the data-to-solution map is Lipschitz continuous.
	If $a$ is real-valued, then we can extend the local solution to a global one $u(t)\in C([0,\I);L^2_x(\R))$ such that its mass is conserved, that is, we have
	\begin{equation}
		\|u(t)\|_{L^2_x(\R)} = \|\phi\|_{L^2_x(\R)}
	\end{equation}
	for all $t\in [0,\I)$.
\end{theorem}

\begin{remark}\label{rmk:def of solution}
	See Section \ref{sec:proof} for the precise statement of Theorem \ref{th:main}.
\end{remark}

\begin{remark}
	The definition of the solution to \eqref{eq:NLS} is a delicate problem.
	One of the standard options is to use the Duhamel integral representation 
	\begin{equation}
		u(t) = S(t)\phi - i \int_0^t S(t-\ta)\K{a(D)(|u(\ta)|^2)u(\ta)}d\ta,
	\end{equation}
	where $S(t)=e^{it\pl_x^2}$.
	However, we cannot expect $a(D)|u(t)|^2$ to make sense as a standard function even if $u(t)$ is a free solution.
	Indeed, when $u$ is a free solution, we only know $|u(t)|^2 \in L^2_t H^{1/2}_x$ by the Ozawa--Tsutsumi type Strichartz estimate (see \eqref{eq:bilinear on I}), so we only expect $a(D)|u|^2 \in L^2_t H^{-1/2+\de}_x$ by the assumption \eqref{eq:assmption for m}.
	This is too rough for the standard Duhamel term.
	Our solution is defined through the closed density equation introduced in Section \ref{sec:novelty} and made rigorous in Definition \ref{def:solution}.
\end{remark}

\begin{remark}
	Theorem \ref{th:main} does not include the ``standard'' derivative NLS
	\begin{equation}\label{eq:standard DNLS}
		i\pl_t u + \pl_x^2 u = \pm i\pl_x(|u|^2 u), \quad u:[0,\I)\times\R \to \C
	\end{equation}
	because the nonlinear term includes $|u|^2 \pl_x u$.
	In this paper, we first find $|u|^2$ and then recover the true solution $u$.
	However, if the nonlinear term includes $|u|^2 \pl_x u$, the reduction to $|u|^2$ does not work well.
	There is a body of literature for \eqref{eq:standard DNLS}. For example, see \cite{Harrop-Griffiths et al 2026} and references therein.
\end{remark}

\subsection{Notation}
We give a list of notations used in this article here.
\begin{itemize}
	\item We define the Fourier and inverse Fourier transforms by
	\begin{equation}
		\CF[f](\xi) = \wh{f}(\xi) := \frac{1}{\sqrt{2\pi}}\int_{\R} e^{-ix\xi} f(x)dx, \quad
		\CF^{-1}[g](x) = \check{g}(x) := \frac{1}{\sqrt{2\pi}}\int_{\R} e^{ix\xi} g(\xi)d\xi.
	\end{equation}
	We define the space-time Fourier transform as $\wt{u}(\ta,\xi):=\CF^{-1}_{t\to\ta}\CF_{x\to\xi} u(t,x)$.
	\item We define the free propagator $S(t)=e^{it\pl_x^2}:=\CF^{-1} e^{-it\xi^2} \CF$.
	We define $S_V(t,s)$ as the propagator of the linear Schr\"odinger equation
	\begin{equation}\label{eq:LSV}
		i\pl_t u +\pl_x^2 u - V(t,x)u=0, \quad u:[0,\I)\times\R_x \to\C.
	\end{equation}
	Namely, $u(t):=S_V(t,s)\phi$ is the solution to \eqref{eq:LSV} with the initial condition $u(s)=\phi$. We write $S_V(t):=S_V(t,0)$.
	\item We denote the operator norm and the Hilbert--Schmidt norm on $L^2(\R)$ by $\|\cdot\|_{\CB}$ and $\|\cdot\|_{\HS}$.
\end{itemize}

\subsection{Novelty and idea of the proof}
\subsubsection{Novelty}
The main contribution of this paper is to develop a linear and multilinear theory for the rough potential $V \in L^2_t H^{-1/2+\de}_x$ from which the well-posedness result follows as an application. More precisely, 
\begin{itemize}
	\item First, we construct the propagator $S_V(t,s)$. For this construction, see Section \ref{sec:UV}.
	\item Second, we extend the Ozawa--Tsutsumi type Strichartz estimate for the free flow
	\begin{equation}\label{eq:Ozawa Tsutsumi 0}
		\No{|S(t)\phi|^2}_{L^2_t([0,T];H_x^{1/2})} \ls_T \|\phi\|_{L^2_x}^2
	\end{equation}
	to the perturbed flow $S_V(t)$ (Proposition \ref{prop:key}).
\end{itemize}

\subsubsection{Idea of the proof}\label{sec:novelty}
Let $\varrho(t):=|u(t)|^2$. Then, at least formally, the solution to \eqref{eq:NLS} satisfies $u(t)=S_{a(D)\varrho}(t)\phi$.
Thus, $\varrho:=|u|^2$ satisfies
\begin{equation}\label{eq:varrho}
	\quad\varrho(t)=|S_{a(D)\varrho}(t)\phi|^2.
\end{equation}
Instead of treating the nonlinear equation directly, we first solve the closed equation of
the density $\varrho$, \eqref{eq:varrho}, and then recover $u$ by the formula $u(t)=S_{a(D)\varrho}(t)\phi$.
This idea has been used in the study of the Hartree equation for infinitely many particles \cite{Chen et al 2018, Lewin Sabin 2014 APDE}.

One of the main difficulties is that, if $\varrho \in L^2_t H^{1/2}_x$, then assumption \eqref{eq:assmption for m} only ensures $V=a(D)\varrho \in L^2_t H^{-1/2+\de}_x$,
so one must first make sense of $S_V(t,s)$ for such rough potentials.
As a first main contribution, this is achieved in Section \ref{sec:UV}
by studying the Dyson series expansion.

In Section 3, we generalize \eqref{eq:Ozawa Tsutsumi 0} to the perturbed flow.
Namely, we prove
$$\No{|S_V(t)\phi|^2}_{L^2_t([0,T];H^{1/2}_x)} \ls \|\phi\|_{L^2_x}^2$$
for small $V \in L^2_t H^{-1/2+\delta}_x$ and $T\le 1$. This is the second main ingredient.

Once these two results are available, the local well-posedness follows by the contraction mapping principle
for the map
$$\Phi[\varrho](t):=|S_{a(D)\varrho}(t)\phi|^2$$
on $L^2_t([0,T];H^{1/2}_x)$. Therefore, the well-posedness theory is not the only output of the
paper; rather, it is derived from the construction of $S_V$ for rough $V$ and from the Strichartz inequality for the perturbed flow.
Finally, when $a$ is real-valued, the $L^2_x$ norm is conserved, and the local solution extends globally.

\section{Construction of $S_V$}\label{sec:UV}
In this section, we consider the linear Schr\"odinger equation
\begin{equation}\label{eq:LS}
	i\pl_t u + \pl_x^2 u - V(t,x)u=0, \quad u:[0,\I)\times\R_x \to \C,
\end{equation}
where $V$ is a given \textit{complex-valued} function.
The aim of this section is to construct $S_V(t,s)$, the propagator of \eqref{eq:LS}.
It will be done in three steps. First, we define $S_V(t,s)$ by the Dyson series expansion under a smallness assumption on $V\in L^2_t H^{-1/2+\de}_x$ (Sections \ref{sec:Dyson} and \ref{sec:smallness}).
Next, we prove stability and semigroup property in this regime (Section \ref{sec:basic}).
Finally, for general $V \in L^2_tH^{-1/2+\de}_x$, we define $S_V(t,s)$ and prove their basic properties (Section \ref{sec:unconditional}). 

\subsection{Dyson series expansion}\label{sec:Dyson}
If $V$ is a sufficiently good function, there is a propagator $S_V(t,s)$ of this equation. More precisely, $S_V(t,s)$ satisfies
\begin{equation}
	S_V(t,s) = S(t-s) - i\int_s^t S(t-\ta)V(\ta)S_V(\ta,s)d\ta.
\end{equation}
Applying this Duhamel formula repeatedly, we obtain the Dyson series expansion
\begin{equation}
	S_V(t,s)= S(t)\sum_{n \ge 0} W_V^{(n)}(t,s)S(s)^*,
\end{equation}
where
\begin{align}
	&W_V^{(0)}(t,s)=\Id_{L^2_x(\R)}, \\
	&W_V^{(n)}(t,s) = (-i)^n \int_s^t dt_1 \int_s^{t_1} dt_2 \cdots \int_s^{t_{n-1}}dt_n S(t_1)^* V(t_1) S(t_1) \cdots S(t_n)^* V(t_n) S(t_n).
\end{align}
Now, we introduce multilinear operators by
\begin{align}
	&W_{V_1,\dots,V_n}^{(n)}(t,s) := (-i)^n \int_s^t dt_1 \int_s^{t_1} dt_2 \cdots \int_s^{t_{n-1}}dt_n S(t_1)^* V_1(t_1) S(t_1) \cdots S(t_n)^* V_n(t_n) S(t_n), \\
	&W_{V_1,\dots,V_n}^{(n),+}(t,s) := (-i)^n \int_s^t dt_1 \int_s^t dt_2 \cdots \int_s^t dt_n S(t_1)^* V_1(t_1) S(t_1) \cdots S(t_n)^* V_n(t_n) S(t_n).
\end{align}
For simplicity, we write
$$W_V^{(n)}(t):=W_V^{(n)}(t,0), \quad W_{V_1,\dots,V_n}^{(n)}(t):= W_{V_1,\dots,V_n}^{(n)}(t,0), \quad W_{V_1,\dots,V_n}^{(n),+}(t):= W_{V_1,\dots,V_n}^{(n),+}(t,0).$$
As you can see in the above definitions, the operator 
\begin{equation}
	\phi \mapsto \int_I S(t)^* (V(t) S(t)\phi) dt
\end{equation}
is the basic building block of the Dyson series expansion.
Thus, the key point is to show that conjugation by the free flow turns the rough potential $V\in L^2_t(I;H^{-1/2+\de}_x)$ into a bounded operator on $L^2_x(\R)$.

\subsection{Short-time construction under a smallness assumption}\label{sec:smallness}
In this section, we give a rigorous definition of $S_V(t,s)$ for small $V\in L^2_tH^{-1/2+\de}_x$ on a short-time interval.
We start with a simple but important observation.
\begin{lemma}\label{lem:1/2 gain}
	Let $I\subset \R$ be an interval.
	Then, we have
	\begin{equation}
		\No{\int_I S(t)^* V(t) S(t)dt}_{\CB} \ls \|V\|_{L^2_t(I;\dot{H}^{-1/2}_x(\R)) + L^{4/3}_t(I;L^2_x(\R))}.
	\end{equation}
\end{lemma}

\begin{proof}
	We can assume $I=\R$.
	First, by the standard Strichartz estimate, we have
	\begin{equation}
		\No{\int_I S(t)^* V(t) S(t)dt}_{\CB} \ls \|V\|_{L^{4/3}_t(I;L^2_x)}.
	\end{equation}
	Next, we prove
	\begin{equation}
		\No{\int_I S(t)^* V(t) S(t)dt}_{\CB} \ls \|V\|_{L^{2}_t(I;\dH^{-1/2}_x)}.
	\end{equation}
	Note that $\|A\|_{\CB} \le \|A\|_{\HS}$ and $\|A\|_{\HS} = \|A(x,y)\|_{L^2_{x,y}}$, where $A(x,y)$ is the integral kernel of $A$.
	Hence, by the unitarity of $\CF$ and $\CF^{-1}$, we have
	\begin{align}
		&\No{\int_\R S(t)^* V(t) S(t)dt}_{\CB} \le \No{\int_\R S(t)^* V(t) S(t)dt}_{\HS} 
		\\
		&\quad\sim \No{\int_\R e^{it(\xi^2-\et^2)} \wh{V}(t,\xi-\et) dt}_{L^2_{\xi,\et}} \sim \|\wt{V}(\xi^2-\et^2,\xi-\et)\|_{L^2_{\xi,\et}},
	\end{align}
	where $\wt{V}:=\CF^{-1}_{t\to\ta} \CF_{x\to\xi} V$.
	By changing variables $c=\xi+\et$ and $r=\xi-\et$, we obtain
	 \begin{align}
	 	\|\wt{V}(\xi^2-\et^2,\xi-\et)\|_{L^2_{\xi,\et}}
	 	\sim \|\wt{V}(rc,r)\|_{L^2_{r,c}} = \||r|^{-1/2} \wt{V}(c,r)\|_{L^2_{r,c}} = \||\pl_x|^{-1/2}V\|_{L^2_{t,x}},
	 \end{align}
	 which completes the proof.
\end{proof}

By Lemma \ref{lem:1/2 gain}, we get the following multilinear estimate.
\begin{lemma}\label{lem:multilinear bound}
	Let $-\I<s<t<\I$ be such that $|t-s|\le 1$.
	Let $0<\de<1$ and $\th_\de = \min(1/4,\de/2)$.
	Then, for any $n \ge 1$, we have
	\begin{equation}\label{eq:multilinear wave operator}
		\No{W_{V_1,\dots,V_n}^{(n)}(t,s)}_{\CB}
		\le C_\de^n |t-s|^{n\th_\de} \prod_{j=1}^n \|V_j\|_{L^2_t([s,t];H^{-1/2+\de}_x(\R))}.
	\end{equation}
\end{lemma}
\begin{proof}
	We can assume $\de\le 1/2$.	
	Let $p_\de = 2/(1+\de)$. 
	By Lemma \ref{lem:1/2 gain}, we have
	\begin{equation}
		\No{W_{V_1,\dots,V_n}^{(n),+}(t,s)}_{\CB} \le C^n \prod_{j=1}^n \|V_j\|_{L^2_t \dot{H}^{-1/2}_x + L^{4/3}_t L^2_x}
		\le C^n \prod_{j=1}^n \|V_j\|_{L^{p_\de}_t \dot{H}^{-1/2+\de}_x + L^{4/3}_t L^2_x} .
	\end{equation}
In particular, we have
	\begin{equation}
	    \No{W_{V_1,\dots,V_n}^{(n),+}(t,s)}_{\CB}
	    \le C^n \prod_{j=1}^n \|V_j\|_{L^{q_\de}_t H^{-1/2+\de}_x},
	\end{equation}
	where $q_\de = \max(p_\de,4/3)$.
	Since $1/q_\de + 1/q_\de >1$, by Lemma \ref{lem:CK}, we have
	\begin{equation}
	\No{W_{V_1,\dots,V_n}^{(n)}(t,s)}_{\CB}
		\le C_\de^n \prod_{j=1}^n \|V_j\|_{L^{q_\de}_t H^{-1/2+\de}_x}\le C_\de^n|t-s|^{n\th_\de} \prod_{j=1}^n \|V_j\|_{L^2_t H^{-1/2+\de}_x},
	\end{equation}
	which completes the proof.
\end{proof}

Since $W_V^{(n)}(t,s)=W_{V,\dots,V}^{(n)}(t,s)$, we have the following corollary.
\begin{corollary}\label{cor:wave operator}
	Let $-\I<s<t<\I$ be such that $|t-s|\le 1$. 
	Let $0<\de<1$ and $\th_\de = \min(1/4, \de/2)$.
	Then, for any $n \ge 1$, we have
	\begin{equation}
		\No{W_V^{(n)}(t,s)}_{\CB} \le C_\de^n |t-s|^{n\th_\de}\|V\|_{L^2_t([s,t];H^{-1/2+\de}_x(\R))}^n.
	\end{equation}
\end{corollary}

Now we give a rigorous definition of $S_V(t,s)$ when $V$ is small.
\begin{definition}[Definition of $S_V$ with a smallness assumption]\label{def:UV}
	Let $-\I<a<b<\I$ be such that $|b-a|\le 1$. 
	Let $0<\de<1$ and $\th_\de=\min(1/4,\de/2)$.
	Assume that $V \in L^2_t([a,b];H^{-1/2+\de}_x(\R))$ satisfies
	\[
	\kappa:=C_\de|b-a|^{\th_\de}\|V\|_{L^2_t([a,b];H^{-1/2+\de}_x(\R))} <1,
	\]
	where $C_\de>0$ is the same constant as in Corollary \ref{cor:wave operator}.
	Then, we define $(S_V(t,s))_{a\le s \le t \le b}$ by
	\begin{equation}\label{eq:UV definition}
		S_V(t,s) = S(t-s)+ S(t)\sum_{n\ge 1}W_V^{(n)}(t,s)S(s)^*.
	\end{equation}
\end{definition}
\begin{remark}
	By Corollary \ref{cor:wave operator} and the smallness assumption for $V$, the infinite series in the right-hand side of \eqref{eq:UV definition} absolutely converges in $\CB(L^2(\R))$. Moreover, we have
	\begin{equation}\label{eq:bound of UV}
		\sup_{a\le s\le t \le b} \|S_V(t,s)\|_{\CB} \le \frac{1}{1-\kappa}.
	\end{equation}
\end{remark}

\subsection{Basic properties}\label{sec:basic}
In this section, we prove that the $S_V$ we constructed in the previous section satisfies reasonable properties.
The next estimate gives the stability of the propagator with respect to the potential. This will be used both to pass from smooth potentials to rough ones and, later, to prove Lipschitz continuity in Section \ref{sec:proof}.
\begin{lemma}[Difference estimate]\label{lem:difference estimate}
Let $-\I<a<b<\I$ be such that $|b-a|\le 1$. 
Let $0<\de<1$ and $\th_\de := \min(1/4,\de/2)$.
Assume that $V,V'\in L^2_t([a,b];H^{-1/2+\de}_x(\R))$ satisfy
$$C_\de|b-a|^{\th_\de}\max\K{\|V\|_{L^2_t([a,b];H^{-1/2+\de}_x(\R))},\|V'\|_{L^2_t([a,b];H^{-1/2+\de}_x(\R))}}\le\frac{2}{3},$$
where $C_\de>0$ is the same constant as in Corollary \ref{cor:wave operator}.
Then, we have
\begin{equation}
	\sup_{a\le s\le t\le b}\No{S_V(t,s)-S_{V'}(t,s)}_{\CB} \le 9C_\de|b-a|^{\th_\de} \|V-V'\|_{L^2_t([a,b];H^{-1/2+\de}_x(\R))}.
\end{equation}
\end{lemma}
\begin{proof}
	By Lemma \ref{lem:multilinear bound}, we obtain
	\begin{align}
		&\|S_V(t,s)-S_{V'}(t,s)\|_{\CB} = \No{\sum_{n\ge 0} W_{V,\dots,V}^{(n)}(t,s)- \sum_{n\ge 0} W_{V',\dots,V'}^{(n)}(t,s)}_{\CB} \\
		&\quad \le \sum_{n\ge 1} \No{W^{(n)}_{V,\dots,V}(t,s) - W_{V',\dots,V'}^{(n)}(t,s)}_{\CB} \\
		&\quad \le \sum_{n\ge 1} \Big(\|W^{(n)}_{V-V',V,\dots,V}(t,s)\|_{\CB} + \|W^{(n)}_{V',V-V',V,\dots,V}(t,s)\|_{\CB} +\cdots + \|W^{(n)}_{V',\dots,V',V-V'}(t,s)\|_{\CB} \Big) \\
		&\quad \le C_\de |t-s|^{\th_\de}\|V-V'\|_{L^2_t([a,b];H^{-1/2+\de}_x)} \sum_{n\ge 1} n\K{\frac{2}{3}}^{n-1}
		\le 9C_\de|b-a|^{\th_\de} \|V-V'\|_{L^2_t([a,b];H^{-1/2+\de}_x)}.
	\end{align}
\end{proof}

Next, we prove some basic properties of $S_V(t,s)$.
\begin{lemma}\label{lem:basic UV}
	Let $-\I<a<b<\I$ be such that $|b-a|\le 1$.
	Let $0<\de<1$ and $\th_\de = \min(1/4,\de/2)$.
	Assume that $V \in L^2_t([a,b];H^{-1/2+\de}_x(\R))$ satisfies
	$$C_\de|b-a|^{\th_\de}\|V\|_{L^2_t([a,b];H^{-1/2+\de}_x(\R))}\le \tw,$$
	where $C_\de>0$ is the same constant as in Corollary \ref{cor:wave operator}.
	Then, we have a family of bounded operators $(S_V(t,s))_{a\le s \le t \le b}$.
	Moreover, it satisfies
	\begin{itemize}
		\item[$(i)$] $S_V(t,s)S_V(s,r) = S_V(t,r)$ for all $a\le r\le s\le t \le b$,
		\item[$(ii)$] $t\mapsto S_V(t,s)$ and $s\mapsto S_V(t,s)$ are strongly continuous in $L^2_x(\R)$.
	\end{itemize}
\end{lemma}

\begin{proof}
	Note that $C_c^\I([a,b]\times\R)$ is dense in $L^2_t([a,b];H_x^{-1/2+\de})$.
	When $V\in C_c^\I([a,b]\times\R)$, the propagator $S_V(t,s)$ satisfies the standard Duhamel formula
	\begin{equation}
		S_V(t,s)=S(t-s) - i\int_s^t S(t-\ta)V(\ta)S_V(\ta,s)d\ta.
	\end{equation}
	Hence, the standard argument implies $(i)$ and $(ii)$.
	
	Let $V \in L^2_t([a,b];H^{-1/2+\de}_x)$. 
	Then, we can choose $V_j\in C_c^\I([a,b]\times\R)$ such that $V_j \to V$ in $L^2_t([a,b];H^{-1/2+\de}_x)$.
	By Lemma \ref{lem:difference estimate}, we know 
	\begin{equation}
		\lim_{j\to \I} \sup_{a\le s\le t\le b} \No{S_{V_j}(t,s) -S_{V}(t,s)}_\CB = 0.
	\end{equation}
	Therefore, $(S_V(t,s))_{a\le s\le t\le b}$ also satisfies $(i)$ and $(ii)$.
\end{proof}

\subsection{Extension to general potentials}\label{sec:unconditional}
Finally, we define $S_V(t,s)$ without any size restriction of $V$ and $|t-s|$.
To remove the smallness assumption, we subdivide the interval $[s,t]$ into subintervals on which the Dyson series converges, and then compose the corresponding short-time propagators.
\begin{definition}[Unconditional definition of $S_V(t,s)$]\label{def:UV unconditional}
	Let $-\I < s < t < \I$. Let $0<\de<1$ and $\th_\de = \min(1/4,\de/2)$.
	Assume that $V \in L^2_t([s,t];H^{-1/2+\de}_x(\R))$. 
	Then, we can decompose $[s,t]$ as $s=r_0 < r_1 < \cdots < r_{N+1} = t$
	such that $|r_{n+1}-r_n|\le 1$ for all $n=0,\dots,N$ and 
	$$\sup_{n=0,\dots,N}C_\de |r_{n+1}-r_{n}|^{\th_\de}\|V\|_{L^2_t([r_{n},r_{n+1}];H^{-1/2+\de}_x(\R))}\le \tw,$$
	where $C_\de>0$ is the same constant as in Corollary \ref{cor:wave operator}.
	Then, we define $S_V(t,s)$ by
	\begin{equation}\label{eq:UV definition unconditional}
		S_V(t,s) = S_V(t,r_{N})\cdots S_V(r_1,s).
	\end{equation}
\end{definition}

\begin{remark}
Definition \ref{def:UV unconditional} is well-defined.
Let $s=r_0<r_1<\cdots<r_{N+1} = t$ and $s=r'_0< \cdots < r'_{N'+1} = t$ be different partitions of $[s,t]$.
Then, we take a common refinement
$$s=r''_0 <r''_1 < \cdots < r''_{N''+1} = t.$$
Therefore, by Lemma \ref{lem:basic UV}, we have 
\begin{equation}
	S_V(t,r_N)\cdots S_V(r_1,s)
	= S_V(t,r''_{N''})\cdots S_V(r''_1,s) = S_V(t,r'_{N'})\cdots S_V(r'_1,s).
\end{equation}
\end{remark}

\begin{lemma}[Basic properties of $S_V$]\label{lem:basic property of UV norm}
	Let $-\I<a<b\le \I$. Let $0<\de<1$. Assume that
	$V \in L^2_{t,\loc} ([a,b);H^{-1/2+\de}_x(\R)).$ Then, we have a family of bounded operators $(S_V(t,s))_{a\le s \le t < b}$.
	Moreover, it satisfies
	\begin{itemize}
	\item[$(i)$] $S_V(t,s)S_V(s,r) = S_V(t,r)$ for any $a\le r\le s\le t< b$,
	\item[$(ii)$] $t\mapsto S_V(t,s)$ and $s \mapsto S_V(t,s)$ are strongly continuous in $L^2_x(\R)$.
	\item[$(iii)$] There exists a monotone increasing function $A:[0,\I)\to [0,\I)$ such that
	\begin{equation}\label{eq:SV bound}
		\|S_V(t,s)\|_{\CB(L^2_x(\R))} \le A\K{\|V\|^2_{L^2_t([s,t];H^{-1/2+\de}_x)} + |t-s|}.
	\end{equation}
	\end{itemize}
\end{lemma}

\begin{proof}
	The properties $(i)$ and $(ii)$ follow from Definition \ref{def:UV unconditional} and Lemma \ref{lem:basic UV}.
	Next we see $(iii)$.
	Let $s=r_0 < \cdots < r_{N+1} =t$ be a partition in Definition \ref{def:UV unconditional}.
	Then, by \eqref{eq:bound of UV}, we obtain
	\begin{equation}
		\|S_V(t,s)\|_{\CB(L^2_x)} \le \prod_{j=0}^N \|S_V(r_{j+1},r_{j})\|_\CB \le 2^{N+1}.
	\end{equation}
	Since we can take $N$ such that
	\[
	N \ls 1+\|V\|_{L^2_t([s,t];H^{-1/2+\de}_x)}^2+|t-s|,
	\]
	we obtain \eqref{eq:SV bound}.
\end{proof}

\section{Bilinear Strichartz estimates for $S_V$}\label{sec:Strichartz}
Recall the well-known bilinear Strichartz estimate proved by Ozawa and Tsutsumi in \cite{Ozawa Tsutsumi 1998}.
\begin{theorem}[Ozawa--Tsutsumi]\label{th:Ozawa Tsutsumi}
	For any $f,g\in L^2_x(\R)$, we have
	\begin{equation}\label{eq:Ozawa Tsutsumi}
		\No{S(t)f \cdot \ov{S(t)g}}_{L^2_t(\R;\dot{H}^{1/2}_x)} \sim \|f\|_{L^2_x(\R)} \|g\|_{L^2_x(\R)}.
	\end{equation}
\end{theorem}
Note that the standard Strichartz estimate implies
\begin{equation}\label{eq:Standard Strichartz}
		\No{S(t)f \cdot \ov{S(t)g}}_{L^2_t(I;L^2_x)} \ls |I|^{1/4} \|f\|_{L^2_x(\R)} \|g\|_{L^2_x(\R)}.
\end{equation}
Collecting \eqref{eq:Ozawa Tsutsumi} and \eqref{eq:Standard Strichartz}, we have
\begin{equation}\label{eq:bilinear on I}
	\No{S(t)f \cdot \ov{S(t)g}}_{L^2_t(I;H^{1/2}_x)} \ls (1+|I|^{1/4})\|f\|_{L^2_x(\R)} \|g\|_{L^2_x(\R)}.
\end{equation}
In the sequel, we generalize \eqref{eq:bilinear on I} to the perturbed flow $S_V$.
Our first observation is
\begin{lemma}\label{lem:multilinear Strichartz}
	Let $0<\de<1$ and $\th_\de = \min(1/4,\de/2)$.
	Let $0<T\le 1$ and $n,m\ge 0$. Then, we have the multilinear estimate
	\begin{equation}\label{eq:mutlilinear Strichartz 2}
		\begin{aligned}
			&\No{S(t)W_{V_1,\dots,V_n}^{(n)}(t)f \cdot \ov{S(t)W_{V'_1,\dots,V'_m}^{(m)}(t) g}}_{L^2_t([0,T];H^{1/2}_x(\R))} \\
			&\quad \le C_\de^{n+m} T^{(n+m)\th_\de}\prod_{j=1}^n \|V_j\|_{L^2_t([0,T];H^{-1/2+\de}_x(\R))}\prod_{k=1}^m \|V_k'\|_{L^2_t([0,T];H^{-1/2+\de}_x(\R))}\|f\|_{L^2_x(\R)} \|g\|_{L^2_x(\R)}.
		\end{aligned}
	\end{equation}
\end{lemma}

\begin{proof}
	We can assume $\de\le 1/2$.
	Let $q_\de = \max(2/(1+\de),4/3)$.
	By Theorem \ref{th:Ozawa Tsutsumi} and Lemma \ref{lem:1/2 gain} (see also the proof of Lemma \ref{lem:multilinear bound}), we have
	\begin{align}
		&\No{S(t)W_{V_1,\dots,V_n}^{(n),+}(T)f \cdot \ov{S(t)W_{V'_1,\dots,V'_m}^{(m),+}(T) g}}_{L^2_t([0,T];H^{1/2}_x)} \\
		&\quad \le C^{n+m} \prod_{j=1}^n \|V_j\|_{L^{q_\de}_t([0,T];H^{-1/2+\de}_x)}\prod_{k=1}^m \|V_k'\|_{L^{q_\de}_t([0,T];H^{-1/2+\de}_x)}\|f\|_{L^2_x} \|g\|_{L^2_x}.
	\end{align}
	Hence, by Lemmas \ref{lem:CK 0} and \ref{lem:CK}, we have
	\begin{align}
		&\No{S(t)W_{V_1,\dots,V_n}^{(n)}(t)f \cdot \ov{S(t)W_{V'_1,\dots,V'_m}^{(m),+}(T) g}}_{L^2_t([0,T];H^{1/2}_x)} \\
		&\quad \le C_\de^{n} C^m \prod_{j=1}^n \|V_j\|_{L^{q_\de}_t([0,T];H^{-1/2+\de}_x)}\prod_{k=1}^m \|V_k'\|_{L^{q_\de}_t([0,T];H^{-1/2+\de}_x)}\|f\|_{L^2_x} \|g\|_{L^2_x}.
	\end{align}
	Again, by Lemmas \ref{lem:CK 0} and \ref{lem:CK}, we obtain
	\begin{align}
		&\No{S(t)W_{V_1,\dots,V_n}^{(n)}(t)f \cdot \ov{S(t)W_{V'_1,\dots,V'_m}^{(m)}(t) g}}_{L^2_t([0,T];H^{1/2}_x)} \\
		&\quad \le C_\de^{n+m} \prod_{j=1}^n \|V_j\|_{L^{q_\de}_t([0,T];H^{-1/2+\de}_x)}\prod_{k=1}^m \|V_k'\|_{L^{q_\de}_t([0,T];H^{-1/2+\de}_x)}\|f\|_{L^2_x} \|g\|_{L^2_x}.
	\end{align}
	Applying the H\"older inequality
	\begin{equation}
		\|f\|_{L^{q_\de}_t([0,T])} \le T^{\de/2} \|f\|_{L^2_t([0,T])} = T^{\th_\de} \|f\|_{L^2_t([0,T])},
	\end{equation}
	we obtain the desired estimate.
\end{proof}

Now we get one of the key estimates in this paper.
\begin{proposition}[Key estimate]\label{prop:key}
	Let $0<\de<1$ and $\th_\de = \min(1/4,\de/2)$.
	Let $0<T\le 1$.
	Assume that $V,V'$ satisfy
	$$C_\de T^{\th_\de} \max\K{\|V\|_{L^2_t([0,T];H^{-1/2+\de}_x(\R))}, \|V'\|_{L^2_t([0,T];H^{-1/2+\de}_x(\R))}}\le \tw,$$
	where $C_\de>0$ is the same constant as in Lemma \ref{lem:multilinear Strichartz}.
	Then, we have
	\begin{align}
		&\No{|S_V(t)\phi|^2}_{L^2_t([0,T];H^{1/2}_x(\R))}
		\le 4\|\phi\|_{L^2_x(\R)}^2 \label{eq:a priori}\\
		&\No{|S_V(t)\phi|^2 - |S_{V'}(t)\phi|^2}_{L^2_t([0,T];H^{1/2}_x(\R))}
		\le 16 C_\de T^{\th_\de}\|V-V'\|_{L^2_t ([0,T];H^{-1/2+\de}_x)}\|\phi\|_{L^2_x}^2.\label{eq:difference}
	\end{align}
\end{proposition}

\begin{remark}
	Estimate \eqref{eq:a priori} is a generalization of \eqref{eq:bilinear on I} for the perturbed propagator, while \eqref{eq:difference} gives the stability with respect to the potential. In Section \ref{sec:proof}, we apply these estimates with $V=a(D)\varrho$ to solve the density equation by a contraction mapping argument. 
\end{remark}

\begin{proof}
By Lemma \ref{lem:multilinear Strichartz}, we have 
\begin{align}
	&\No{|S_V(t)\phi|^2}_{L^2_t([0,T];H^{1/2}_x)}
	       \le \sum_{n,m\ge 0} \No{S(t)W_V^{(n)}(t)\phi \cdot \ov{S(t)W_V^{(m)}(t)\phi}}_{L^2_t([0,T];H^{1/2}_x)} \\
	&\quad \le \sum_{n,m\ge 0} C_\de^{n+m} T^{(n+m)\th_\de}\|V\|_{L^2_t([0,T];H^{-1/2+\de}_x)}^{n+m} \|\phi\|_{L^2_x}^2 \le 4\|\phi\|_{L^2_x}^2,
\end{align}
which proves \eqref{eq:a priori}.
Next, we have
\begin{align}
	&\No{|S_V(t)\phi|^2 - |S_{V'}(t)\phi|^2}_{L^2_t([0,T];H^{1/2}_x)} \\
	&\quad \le \No{S_V(t)\phi \cdot \ov{(S_V(t)-S_{V'}(t))\phi}}_{L^2_t([0,T];H^{1/2}_x)}
	    + \No{(S_V(t) - S_{V'}(t))\phi \cdot \ov{S_{V'}(t)\phi}}_{L^2_t([0,T];H^{1/2}_x)} \\
	&\quad =:\SA+\SB.
\end{align}
We only estimate $\SA$ because we can deal with $\SB$ in the same way.
By Lemma \ref{lem:multilinear Strichartz}, we have
\begin{align}
	\SA&\le \sum_{n\ge 0, m\ge 1} \No{S(t)W_{V,\dots,V}^{(n)}(t)\phi\cdot\ov{S(t)(W_{V,\dots,V}^{(m)}(t) - W_{V',\dots,V'}^{(m)}(t))\phi}}_{L^2_t([0,T];H^{1/2}_x)} \\
	&\le \sum_{n\ge 0, m\ge 1} \bigg(\No{S(t)W_{V,\dots,V}^{(n)}(t)\phi\cdot\ov{S(t)W_{V-V',V,\dots,V}^{(m)}(t)\phi}}_{L^2_t([0,T];H^{1/2}_x)} \\
	&\quad + \No{S(t)W_{V,\dots,V}^{(n)}(t)\phi\cdot\ov{S(t)W_{V',V-V',V,\dots,V}^{(m)}(t)\phi}}_{L^2_t([0,T];H^{1/2}_x)} + \cdots \\
	&\quad\cdots  + \No{S(t)W_{V,\dots,V}^{(n)}(t)\phi\cdot\ov{S(t)W_{V',\dots,V',V-V'}^{(m)}(t)\phi}}_{L^2_t([0,T];H^{1/2}_x)}\bigg) \\
	&\le C_\de T^{\th_\de}\|V-V'\|_{L^2_t([0,T];H^{-1/2+\de}_x)} \|\phi\|_{L^2_x}^2 \sum_{n\ge 0, m\ge 1} \frac{m}{2^{n+m-1}} \\
	&= 8C_\de T^{\th_\de} \|V-V'\|_{L^2_t([0,T];H^{-1/2+\de}_x)} \|\phi\|_{L^2_x}^2.
\end{align}
\end{proof}

\section{Proof of the well-posedness result}\label{sec:proof}
\subsection{Local well-posedness}
In this section, we give a proof of Theorem \ref{th:main}.
First, we give a rigorous definition of the solution.
Since the usual Duhamel formulation is not available, we do not define the solution directly through the nonlinear equation. Instead, we first solve the closed equation for the density
$$\varrho(t) = |S_{a(D)\varrho}(t)\phi|^2.$$
Then, we  recover the true solution by $u(t)=S_{a(D)\varrho}(t)\phi$.
This argument has been used in the context of the Hartree equation for infinitely many particles. For example, see \cite{Lewin Sabin 2014 APDE, Chen et al 2018}.

\begin{definition}[Definition of the solution to \eqref{eq:NLS}]\label{def:solution}
	Let $-\I< t_0 < t_1 \le \I$ and $I=[t_0,t_1)$.
	Let $0<\de<1$.
	We call $u(t) \in C(I;L^2_x(\R))$ a solution to \eqref{eq:NLS} with initial condition $u(t_0)=\phi$ if the following holds: There exists $\varrho \in L^2_{t,\loc}(I;H^{1/2}_x(\R))$ such that
	$$\varrho(t) = |S_{a(D)\varrho}(t,t_0)\phi|^2 \quad \text{in } L^2_{t,\loc}([t_0,t_1);H^{1/2}_x(\R)),$$
	and $u(t)\in C(I;L^2_x(\R))$ is given by $u(t)=S_{a(D)\varrho}(t,t_0)\phi$.
\end{definition}

\begin{remark}
	Since $\varrho \in L^2_{t,\loc}(I;H^{1/2}_x(\R))$, \eqref{eq:assmption for m} implies $a(D)\varrho \in L^2_{t,\loc}(I;H^{-1/2+\de}_x(\R))$.
	Therefore, the propagator $(S_{a(D)\varrho}(t,s))_{t_0\le s\le t <t_1}$ is well-defined.
	Moreover, $S_{a(D)\varrho}(t,t_0)$ is strongly continuous with respect to $t$, and $u:I \to L^2_x(\R)$ is also continuous.
\end{remark}

\begin{remark}
	It is natural to ask how Definition \ref{def:solution} is related to the classical notion of solution.
	Since $u=S_{a(D)|u|^2}(t)\phi$, $u(t)$ satisfies
	\begin{equation}
		u(t) = S(t)\sum_{n\ge 0} W_{a(D)|u|^2}^{(n)}(t) \phi \overset{\textup{formally}}{=} S(t)\phi - i\int_0^t S(t-\ta)\K{a(D)|u(\ta)|^2 u(\ta)}d\ta,
	\end{equation}
	where the second equality holds only formally in general because the last term does not always make sense.
	However, if the initial data $\phi$ is sufficiently smooth, then $u$ and $a(D)|u|^2$ are also nice functions; hence the last formal equality actually holds.
	Thus, Definition \ref{def:solution} is consistent with the standard Duhamel integral formulation whenever the latter is available.
\end{remark}

Now we prove local well-posedness.
\begin{proposition}[Local well-posedness]\label{prop:LWP}
	Assume that $a:\R\to \C$ satisfies \eqref{eq:assmption for m} with $0<\de<1$.
	Then, for any $\phi \in L^2_x(\R)$, there exists $T=T(a,\de,\|\phi\|_{L^2_x(\R)}) > 0$ such that the following hold.
	\begin{itemize}
		\item[$(i)$] There exists a unique solution $u(t)\in C([0,T];L^2_x(\R))$ to \eqref{eq:NLS} with initial condition $u(0)=\phi$ such that
		$|u(t)|^2 \in L^2_t([0,T];H^{1/2}_x(\R))$.
		\item[$(ii)$] For any $R>0$ and $T=T(a,\de,R)>0$, the data-to-solution map
		\begin{equation}
			\{\ph\in L^2(\R):\|\ph\|_{L^2(\R)} \le R\} \ni \phi \to (u,\varrho) \in C([0,T];L^2_x(\R))\times L^2_t([0,T];H^{1/2}_x(\R))   
		\end{equation}
		is Lipschitz continuous.
		\item[$(iii)$] If $a$ is a real-valued function, we have the mass conservation law, that is, 
		$$\|u(t)\|_{L^2_x(\R)} = \|\phi\|_{L^2_x(\R)}$$
		for all $t\in[0,T]$.
	\end{itemize}
\end{proposition}

\begin{proof}
	\noindent\textbf{Step 1: Unique existence of the local-in-time solution.}
Let $R:=4\|\phi\|_{L^2_x}^2+1$ and $T=T(a,\de,\|\phi\|_{L^2_x})>0$ be sufficiently small such that
\begin{equation}
	16 C_\de T^{\th_\de} R \sup_{\xi\in \R} \frac{|a(\xi)|}{\lxr^{1-\de}} \le \tw, \quad T\le 1,
\end{equation}
where $C_\de$ is a constant depending only on $\de>0$.
Let $\Ph[\varrho](t):=|S_{a(D)\varrho}(t)\phi|^2$.
	Define $E(T,R)$ by
	$$E(T,R):=\Ck{\varrho: \|\varrho\|_{L^2_t([0,T];H^{1/2}_x)} \le R}.$$
	Let $\varrho \in E(T,R)$.
	Since
	\begin{equation}
		C_\de T^{\th_\de} \|a(D)\varrho\|_{L^2_t([0,T];H^{-1/2+\de}_x)} \le 
		C_\de T^{\th_\de}\sup_{\xi\in\R}\frac{|a(\xi)|}{\lxr^{1-\de}} \|\varrho\|_{L^2_t([0,T];H^{1/2}_x)} \le \tw,
	\end{equation}
	Proposition \ref{prop:key} implies
	\begin{equation}
		\|\Ph[\varrho]\|_{L^2_t([0,T];H^{1/2}_x)} \le 4 \|\phi\|_{L^2_x}^2 \le R.
	\end{equation}
	Therefore, the map $\Ph:E(T,R) \to E(T,R)$ is well-defined.
	Similarly, by Proposition \ref{prop:key}, we have
	\begin{align}
		&\|\Ph[\varrho]-\Ph[\varrho']\|_{L^2_t([0,T];H^{1/2}_x)}
		\le 16 C_\de T^{\th_\de}\|a(D)\varrho - a(D)\varrho'\|_{L^2_t([0,T];H^{-1/2+\de}_x)} \|\phi\|_{L^2_x}^2 \\
		&\quad \le 16 C_\de T^{\th_\de} R \sup_{\xi\in \R} \frac{|a(\xi)|}{\lxr^{1-\de}} \|\varrho-\varrho'\|_{L^2_t([0,T];H^{1/2}_x)}
		\le \tw \|\varrho-\varrho'\|_{L^2_t([0,T];H^{1/2}_x)}.
	\end{align}
	Hence, $\Ph:E(T,R)\to E(T,R)$ is a contraction mapping, and the fixed point theorem ensures that there exists at least one solution $\varrho \in L^2_t([0,T];H^{1/2}_x)$.
	The uniqueness of $\varrho$ follows from the standard argument.
	Hence, we obtain the true solution $u(t):=S_{a(D)\varrho}(t)\phi$.
	The uniqueness of $\varrho$ implies the uniqueness of $u$.
	
	\noindent\textbf{Step 2: Lipschitz continuity.}
	We can prove the Lipschitz continuity of data-to-density map
	\begin{equation}
		\Ck{\ph:\|\ph\|_{L^2_x}\le R} \ni \phi \mapsto \varrho \in L^2_t([0,T];H^{1/2}_x)
	\end{equation}
	by the standard argument.
	Let $\phi,\phi' \in \{\ph:\|\ph\|_{L^2_x}\le R\}$ and $\varrho,\varrho'\in L^2_t([0,T];H^{1/2}_x)$ be corresponding solutions.
	Let $u(t)=S_{a(D)\varrho}(t)\phi$ and $u'(t)=S_{a(D)\varrho'}(t)\phi'$.
	Then, by Lemma \ref{lem:difference estimate} and Lipschitz continuity of data-to-density map, we have
	\begin{align}
		&\sup_{t\in [0,T]}\|S_{a(D)\varrho}(t)\phi - S_{a(D)\varrho'}(t)\phi'\|_{L^2_x} \\
		&\quad \le \sup_{t\in [0,T]}\|S_{a(D)\varrho}(t)-S_{a(D)\varrho'}(t)\|_{\CB} \|\phi\|_{L^2_x}
		         + \sup_{t\in[0,T]} \|S_{a(D)\varrho'}(t)\|_{\CB} \|\phi-\phi'\|_{L^2_x} \\
		&\quad \le 9 C_\de T^{\th_\de} R \sup_{\xi\in\R}\frac{|a(\xi)|}{\lxr^{1-\de}} \|\varrho-\varrho'\|_{L^2_t([0,T];H^{1/2}_x)} + 2 \|\phi-\phi'\|_{L^2_x}
		\ls_{a,\de,R} \|\phi-\phi'\|_{L^2_x},
	\end{align}
	where we used \eqref{eq:bound of UV}.
	
	\noindent \textbf{Step 3: Conservation of $L^2(\R)$ norm.}
	Note that $C_c^\I([0,T]\times\R) \subset L^2_t([0,T];H^{-1/2+\de}_x)$ is a dense subset.
	Let $V := a(D)|u|^2 \in L^2_t([0,T];H^{-1/2+\de}_x)$. Then, we can choose $V_j \in C_c^\I([0,T]\times\R)$ such that 
	$V_j \to V$ in $L^2_t([0,T];H^{-1/2+\de}_x)$ as $j\to \I$. 
	Let $u_j(t):=S_{V_j}(t)\phi$. Then, by the standard argument, we have
	\begin{equation}
		\|u_j(t)\|_{L^2_x}^2 = \|\phi\|_{L^2_x}^2 + 2\Im\Dk{\int_0^t \int_\R \sd^{-1/2+\de}V_j(\ta,x) \cdot \sd^{1/2-\de} |u_j(\ta,x)|^2 dxd\ta}.
	\end{equation}
	Since the difference estimate \eqref{eq:difference} implies
	\begin{equation}
		|u_j|^2 \to |u|^2 \quad \text{in } L^2_t([0,T];H^{1/2}_x) \text{ as } j \to \I,
	\end{equation}
	we obtain
	\begin{equation}
		\|u(t)\|_{L^2_x}^2 = \|\phi\|_{L^2_x}^2 + 2\Im\Dk{\int_0^t \int_\R \sd^{-1/2+\de}V(\ta,x) \cdot \sd^{1/2-\de} |u(\ta,x)|^2 dxd\ta}.
	\end{equation}
	Finally, since $a(\xi)$ is real-valued, we have
	\begin{align}
		&\int_\R \sd^{-1/2+\de}V(\ta,x) \cdot \sd^{1/2-\de} |u(\ta,x)|^2 dx \\
		&\quad	= \int_\R \sd^{-1/2+\de} a(D)\varrho (\ta,x) \cdot \sd^{1/2-\de} \varrho(\ta,x) dx
		        = \int_\R \frac{a(\xi)}{\lxr^{1-\de}} \lxr^{1-\de}|\wh{\varrho}(\ta,\xi)|^2 d\xi \in \R,
	\end{align}
	which implies $\|u(t)\|_{L^2_x}^2 = \|\phi\|_{L^2_x}^2$ for all $t \in [0,T]$.
\end{proof}

\subsection{Global well-posedness}
In this section, we assume $a$ is real-valued and extend the solution that we obtained in Proposition \ref{prop:LWP}.
Since our notion of solution is defined through the density and the propagator $S_{a(D)\varrho}$, it is not obvious whether we can extend a solution. The next lemma shows that solutions on adjacent intervals are compatible, thanks to the semigroup property (Lemma \ref{lem:basic property of UV norm}).
\begin{lemma}\label{lem:extension}
	Let $0\le t_0 <t_1<t_2<\I$.
	Let $u_0(t)=S_{a(D)\varrho_0}(t,t_0)\phi \in C([t_0,t_1];L^2_x(\R))$ and $u_1(t) = S_{a(D)\varrho_1}(t,t_1)\ps \in C([t_1,t_2];L^2_x(\R))$ be two solutions to \eqref{eq:NLS}.
	If $u_0(t_1) = u_1(t_1) = \ps$, then $u(t)\in C([t_0,t_2];L^2_x(\R))$ defined by
	\begin{equation}
		u(t) = \begin{dcases}
			u_0(t) \quad&\text{if } t_0 \le t \le t_1,\\
			u_1(t) \quad&\text{if } t_1 \le t \le t_2
		\end{dcases}
	\end{equation}
	is also a solution to \eqref{eq:NLS}.
\end{lemma}

\begin{proof}
	Define $\varrho \in L^2_t([t_0,t_2];H^{1/2}_x)$ by
	\begin{equation}
		\varrho(t) := \begin{dcases}
			\varrho_0(t) \quad&\text{if } t_0 \le t \le t_1,\\
			\varrho_1(t) \quad&\text{if } t_1 \le t \le t_2.
		\end{dcases}
	\end{equation}
	Then, we have $S_{a(D)\varrho}(t,t_0) = S_{a(D)\varrho_0}(t,t_0)$ for all $t_0\le t\le t_1$
	and $S_{a(D)\varrho}(t,t_1) = S_{a(D)\varrho_1}(t,t_1)$ for all $t_1\le t\le t_2$.
	On the one hand, we have
	$$\varrho(t) = \varrho_0(t) =  |S_{a(D)\varrho_0}(t,t_0)\phi|^2 = |S_{a(D)\varrho}(t,t_0)\phi|^2 \quad \text{in } L^2_t([t_0,t_1];H^{1/2}_x).$$
	On the other hand, by Lemma \ref{lem:basic property of UV norm}, we have
	\begin{align}
	\varrho(t) &= \varrho_1(t) =  |S_{a(D)\varrho_1}(t,t_1)\ps|^2 = |S_{a(D)\varrho_1}(t,t_1)S_{a(D)\varrho_0}(t_1,t_0)\phi|^2 \\
	&= |S_{a(D)\varrho}(t,t_1)S_{a(D)\varrho}(t_1,t_0)\phi|^2 = |S_{a(D)\varrho}(t,t_0)\phi|^2 \quad \text{in } L^2_t([t_1,t_2];H^{1/2}_x).
	\end{align}
	Therefore, $\varrho\in L^2_t([t_0,t_2];H^{1/2}_x)$ satisfies
	$$\varrho = |S_{a(D)\varrho}(t,t_0)\phi|^2$$
	and $u'(t):= S_{a(D) \varrho}(t,t_0)\phi \in C([t_0,t_2];L^2_x(\R))$ is also a solution to \eqref{eq:NLS} with the initial condition $u'(t_0)=\phi$.
	Finally, we have 
	$$u'(t)=S_{a(D)\varrho}(t,t_0)\phi = S_{a(D)\varrho_0}(t,t_0)\phi=u_0(t)=u(t)$$
	for all $t\in [t_0,t_1]$, and 
	$$u'(t)=S_{a(D)\varrho}(t,t_0)\phi = S_{a(D)\varrho_1}(t,t_1)S_{a(D)\varrho_0}(t_1,t_0)\phi=u_1(t) = u(t)$$
	for all $t\in [t_1,t_2]$.
	Therefore, we complete the proof.
\end{proof}

Now, we give a proof of Theorem \ref{th:main}.
\begin{proof}
Let $\phi \in L^2_x$.
By Proposition \ref{prop:LWP}, we have $T=T(a,\de,\|\phi\|_{L^2_x(\R)})>0$ and a unique solution $u(t)\in C([0,T];L^2_x)$, which satisfies
$\|u(t)\|_{L^2_x} = \|\phi\|_{L^2_x}$ for all $t\in[0,T]$.
Applying the same argument in the proof of Proposition \ref{prop:LWP}, we obtain a unique solution $v(t)\in C([T,2T];L^2_x)$ such that $v(T)=u(T)$, which satisfies $\|v(t)\|_{L^2_x} = \|\phi\|_{L^2_x}$ for all $t\in [T,2T]$.
Therefore, by Lemma \ref{lem:extension}, we obtain a unique extended solution $u(t) \in C([0,2T];L^2_x)$.
Repeating the same argument, we obtain the desired global-in-time solution $u(t)\in C([0,\I);L^2_x)$.
\end{proof}

\appendix

\section{Christ--Kiselev type lemmas}
In this section, we present the Christ--Kiselev type lemmas.
First, recall the classical Christ--Kiselev lemma \cite{Christ Kiselev 2001}.
\begin{lemma}[Christ--Kiselev]\label{lem:CK 0}
	Let $X,Y$ be Banach spaces.
	Let $-\I \le a < b \le \I$.
	Define linear operators $T,T_<$ by
	\begin{align}
		T[g](t):= \int_a^b K(t,s)g(s)ds, \quad T_< [g](t) := \int_a^t K(t,s)g(s)ds.
	\end{align}
	If $p<q$, then we have
	\begin{align}
		&\|T[g]\|_{L^q_t((a,b);Y)} \le C \|g\|_{L^p_t((a,b);X)}
		\implies \|T_<[g]\|_{L^q_t((a,b);Y)} \le C C_{p,q} \|g\|_{L^p_t((a,b);X)}.
	\end{align}
\end{lemma}

In this paper, we often use the dual version of Lemma \ref{lem:CK 0}.
Its proof is well-known, but we give it below for convenience.
\begin{lemma}\label{lem:CK}
	Let $-\I\le a<b\le \I$.
	Let $X,Y,Z$ be Banach spaces. Assume that, for each $t,s\in(a,b)$, $K(t,s;f,g)$ is bilinear with respect to $(f,g)$.
	Define bilinear operators $B,B_<$ by
	\begin{align}
		&B[f,g]= \int_a^b dt \int_a^b ds K(t,s;f(t),g(s)),\\
		&B_< [f,g] = \int_a^b dt \int_a^{t} ds K(t,s;f(t),g(s)).
	\end{align}
	If $1/p + 1/q >1$, then we have
	\begin{align}
		&\|B[f,g]\|_Z \le C \|f\|_{L^{q}_t((a,b);Y)} \|g\|_{L^{p}_t((a,b);X)}
		 \implies \|B_<[f,g]\|_Z \le C C_{p,q} \|f\|_{L^{q}_t((a,b);Y)} \|g\|_{L^{p}_t((a,b);X)}.
	\end{align}
\end{lemma}

\begin{proof}[Proof of Lemma \ref{lem:CK}]
We will follow \cite[Proof of Lemma 3.1]{Tao 2000}.
By the bilinearity, it suffices to prove
\begin{equation}
 \|B_<[f,g]\|_{Z} \le C C_{p,q}<\I	\text{ when } \|f\|_{L^{q}_t((a,b);Y)}^q = \|g\|_{L^{p}_t((a,b);X)}^p = \tw.
\end{equation}
Let $\ph:(a,b) \to (0,1)$ be a smooth and strictly increasing function such that $\ph(a+) = 0$ and $\ph(b-)=1$.
Define a strictly increasing bijection $F:(a,b) \to (0,1)$ by
\begin{align}
	F(t) := \frac{1}{2}\Ck{\ph(t) + \|f\|_{L^{q}((a,t);Y)}^q + \|g\|_{L^{p}((a,t);X)}^p}.
\end{align}
For dyadic intervals $I,J\subset[0, 1]$, we write $I\sim J$ when $I$ and $J$ have the same length, are not adjacent but have adjacent parents, and $I$ is on the right of $J$.
Then, we have
\begin{equation}
	\II_{(s<t)} = \sum_{I\sim J} \II_{F^{-1}(I)}(t) \II_{F^{-1}(J)}(s).
\end{equation}
Since
\begin{align}
	B_<[f,g]&= \int_a^b dt \int_a^{t} ds K(t,s;f(t),g(s)) \\
	&= \sum_{I\sim J} \int_a^b dt \int_a^b ds K(t,s;\II_{F^{-1}(I)}(t)f(t),\II_{F^{-1}(J)}(s)g(s)),
\end{align}
we obtain
\begin{align}
	\|B_<[f,g]\|_{Z}
	&\le \sum_{I\sim J} \|B(\II_{F^{-1}(I)}(t) f, \II_{F^{-1}(J)}(s) g)\|_Z \\
	&\le C \sum_{I\sim J}  \|\II_{F^{-1}(I)}(t)f\|_{L^{q}_t((a,b);Y)}
	       \|\II_{F^{-1}(J)}(s)g\|_{L^{p}_t((a,b);X)}.
\end{align}
Changing variables, we obtain
\begin{equation}
	\|\II_{F^{-1}(I)}(t)f\|_{L^{q}_t((a,b);Y)} \ls |I|^{1/q}, 
	\quad \|\II_{F^{-1}(J)}(t) g\|_{L^{p}_t((a,b);X)} \ls |J|^{1/p}.
\end{equation}
Thus, we have
\begin{align}
	\|B_<[f,g]\|_{Z} &\le C \sum_{I\sim J} |I|^{1/q}|J|^{1/p} \le C\sum_{k=1}^\I \sum_{I\sim J, \, |I|=2^{-k}} |I|^{1/q}|J|^{1/p}\\
	&\le C \sum_{k=1}^\I 2^{-k(1/p+1/q-1)} = C C_{p,q} < \I,
\end{align}
where we used $1/p+1/q-1 >0$ by the assumption.
\end{proof}

\end{document}